\documentclass[10pt]{amsart}
\usepackage{amssymb}
\usepackage{tikz-cd}
\usepackage{amsrefs}

\title{Abstract embeddability ranks}

\author {Florent P. Baudier}
\address{Department of Mathematics\\Texas A\&M University\\College Station\\Texas\\USA}
\email{florent@tamu.edu}
\urladdr{https://www.math.tamu.edu/~florent/}

\author {Christian Rosendal}
\address{Department of Mathematics\\University of Maryland\\4176 Campus Drive - William E. Kirwan Hall\\College Park, MD 20742-4015\\USA}
\email{rosendal@umd.edu}
\urladdr{sites.google.com/view/christian-rosendal/}

\newcommand{\norm}[1]{\lVert#1\rVert}

\newcommand{\forkindep}[1][]{\mathop{\mathop{\vcenter{\hbox{\oalign{\noalign{\kern-.3ex}
\hfil$\vert$\hfil\cr\noalign{\kern-.7ex}$\smile$\cr\noalign{\kern-.3ex}}}}}\displaylimits_{#1}}}
\newcommand{\maths}[1]{\[\begin{split}{#1}\end{split}\]}
\newcommand{\maps}[1]{\mathop{\overset{#1}\longrightarrow}}
\newcommand {\N}{\mathbb N}
\newcommand {\Z}{\mathbb Z}
\newcommand {\Q}{\mathbb Q}

\newcommand{\om}{\omega}

\newcommand{\tom} {\emptyset}

\newcommand{\saa}{\Rightarrow}
\newcommand{\equi}{\Leftrightarrow}

\newcommand {\Del}{ \; \Big| \;}
\newcommand {\del}{ \; \big| \;}

\newcommand {\go} {\mathfrak}
\newcommand {\ku} {\mathcal}

\newcommand {\e} {\exists}

\theoremstyle{plain}
\newtheorem{thm}{Theorem}[section]
\newtheorem*{theorem*}{Theorem}
\newtheorem{cor}[thm]{Corollary}
\newtheorem{lemme}[thm]{Lemma}
\newtheorem{prop} [thm] {Proposition}

\theoremstyle{definition}

\definecolor{groen}{rgb}{0,0.5,.7}
\definecolor{gul}{rgb}{0.94,0.8,0}
\definecolor{blaa}{rgb}{0.16,0,0.6}
\definecolor{roed}{rgb}{1,0,0}

\makeindex

\thanks{The authors were partially supported by the U.S. National Science Foundation under Grant Numbers DMS-2055604 (F.B.) and DMS-2204849 (C.R.).}

\begin{document}
%\subjclass[2010]{Primary: 20E08, Secondary: 03E15}

\maketitle

%%%%%%%%%%%%%%%%%%%%%%%%%%%%%%%%%%%%%%%%%%%%%%%%%%%%%%%%%%%
%%%%%%%%%%%%%%%%%%%%%%%%%%%%%%%%%%%%%%%%%%%%%%%%%%%%%%%%%%%
%%%%%%%%%%%%%%%%%%%%%%%%%%%%%%%%%%%%%%%%%%%%%%%%%%%%%%%%%%%
%%%%%%%%%%%%%%%%%%%%%%%%%%%%%%%%%%%%%%%%%%%%%%%%%%%%%%%%%%%
%%%%%%%%%%%%%%%%%%%%%%%%%%%%%%%%%%%%%%%%%%%%%%%%%%%%%%%%%%%
%%%%%%%%%%%%%%%%%%%%%%%%%%%%%%%%%%%%%%%%%%%%%%%%%%%%%%%%%%%
%%%%%%%%%%%%%%%%%%%%%%%%%%%%%%%%%%%%%%%%%%%%%%%%%%%%%%%%%%%
%%%%%%%%%%%%%%%%%%%%%%%%%%%%%%%%%%%%%%%%%%%%%%%%%%%%%%%%%%%
%%%%%%%%%%%%%%%%%%%%%%%%%%%%%%%%%%%%%%%%%%%%%%%%%%%%%%%%%%%
%%%%%%%%%%%%%%%%%%%%%%%%%%%%%%%%%%%%%%%%%%%%%%%%%%%%%%%%%%%
%%%%%%%%%%%%%%%%%%%%%%%%%%%%%%%%%%%%%%%%%%%%%%%%%%%%%%%%%%%
%%%%%%%%%%%%%%%%%%%%%%%%%%%%%%%%%%%%%%%%%%%%%%%%%%%%%%%%%%%
%%%%%%%%%%%%%%%%%%%%%%%%%%%%%%%%%%%%%%%%%%%%%%%%%%%%%%%%%%%
%%%%%%%%%%%%%%%%%%%%%%%%%%%%%%%%%%%%%%%%%%%%%%%%%%%%%%%%%%%
%%%%%%%%%%%%%%%%%%%%%%%%%%%%%%%%%%%%%%%%%%%%%%%%%%%%%%%%%%%
%%%%%%%%%%%%%%%%%%%%%%%%%%%%%%%%%%%%%%%%%%%%%%%%%%%%%%%%%%%

\begin{abstract}
We describe several ordinal indices that are capable of detecting, according to various metric notions of faithfulness, the embeddability between pairs of Polish spaces. These embeddability ranks are of theoretical interest but seem difficult to estimate in practice. Embeddability ranks, which are easier to estimate in practice, are embeddability ranks generated by Schauder bases. These embeddability are inspired by the nonlinear indices \`a la Bourgain from \cite{BLMS_FM}. In particular, we resolve a problem \cite[Problem 3.10]{BLMS_FM} regarding the necessity of additional set-theoretic axioms regarding the main coarse universality result of \cite{BLMS_FM}.  
\end{abstract}

%%%%%%%%%%%%%%%%%%%%%%%%%%%%%%%%%%%%%%%%%%%%%%%%%%%%%%%%%%%
%%%%%%%%%%%%%%%%%%%%%%%%%%%%%%%%%%%%%%%%%%%%%%%%%%%%%%%%%%%
%%%%%%%%%%%%%%%%%%%%%%%%%%%%%%%%%%%%%%%%%%%%%%%%%%%%%%%%%%%
%%%%%%%%%%%%%%%%%%%%%%%%%%%%%%%%%%%%%%%%%%%%%%%%%%%%%%%%%%%
%%%%%%%%%%%%%%%%%%%%%%%%%%%%%%%%%%%%%%%%%%%%%%%%%%%%%%%%%%%
%%%%%%%%%%%%%%%%%%%%%%%%%%%%%%%%%%%%%%%%%%%%%%%%%%%%%%%%%%%
%%%%%%%%%%%%%%%%%%%%%%%%%%%%%%%%%%%%%%%%%%%%%%%%%%%%%%%%%%%
%%%%%%%%%%%%%%%%%%%%%%%%%%%%%%%%%%%%%%%%%%%%%%%%%%%%%%%%%%%
%%%%%%%%%%%%%%%%%%%%%%%%%%%%%%%%%%%%%%%%%%%%%%%%%%%%%%%%%%%
%%%%%%%%%%%%%%%%%%%%%%%%%%%%%%%%%%%%%%%%%%%%%%%%%%%%%%%%%%%

\section{Introduction}
In 1980, Bourgain \cite{Bourgain1980} showed that a separable Banach space that contains an isomorphic copy of every separable \emph{reflexive} space must also contain an isomorphic copy of $C[0,1]$, and thus, by Banach-Mazur embedding theorem, an isomorphic copy of every separable Banach space. This universality theorem is a far reaching improvement of an earlier result of Szlenk \cite{Szlenk1968}, which states that there is no separable reflexive Banach space that is isomorphically universal for the class of separable reflexive Banach spaces. Bourgain's argument in turn relies on a reformulation of the problem in terms of the computation of an ordinal index, making the problem amenable to descriptive set theoretic techniques. Bourgain's innovative approach was further developed in \cite{Bossard2002}, \cite{ArgyrosDodos}, and \cite{Dodos2009}, in order to show that a class of Banach spaces that is analytic, in the Effros-Borel structure of subspaces of $C[0,1]$, and contains all separable reflexive Banach spaces, must contain a universal space.
In \cite{BLMS_FM}, nonlinear versions of Bourgain's universality theorem were proposed. Thanks to Aharoni's theorem \cite{Aharoni1974}, it is known that $c_0$ contains a bi-Lipschitz, and hence a coarse and uniform, copy of every separable metric space. Depending on whether we consider bi-Lipschitz embeddings or coarse embeddings, the role of the class of reflexive spaces considered by Szlenk or Bourgain is played by a collection of metric spaces $R[\ku S_\alpha]_{c_0}$ with ${\alpha<\omega_1}$  and where $R=\Q$ for bi-Lipschitz embeddings and $R=\Z$ for coarse embeddings. Here $R[\ku S_\alpha]_{c_0}$ is the metric subspace of $c_0$ consisting of all vectors
$\sum_{i\in A}x_ie_i$
where $x_i\in R$ and $A$ belong to the $\alpha^{\rm{th}}$ Schreier family $\ku S_\alpha$.

In this note we prove the coarse universality result below. This result was originally proved in \cite{BLMS_FM} under  additional set theoretical assumptions, namely Martin's Axiom (MA) and the negation of the Continuum Hypothesis ($\neg\text{CH}$). The main point of our theorem is to show that these axioms are actually not required, thus solving Problem 3.10 from \cite{BLMS_FM}.

\begin{theorem*}\label{T:B}
If a separable metric space contains a coarse copy of $\Z[\ku S_\alpha]_{c_0}$ for every ordinal $\alpha<\omega_1$, then it contains a coarse copy of every separable metric space.
\end{theorem*}

In \cite{BLMS_FM}, the proof (under MA+$\neg$CH) of the theorem above relied on a reformulation \`a la Bourgain of the universality problem. In particular, certain combinatorial trees on Polish spaces capturing the nonlinearity of the problem together with associated ordinal indices were introduced. In order for this approach to be successful in the coarse setting, it was necessary to extract an uncountable family of \emph{equi}-coarse embeddings out of an uncountable family of merely coarse embeddings. While the similar problem in the Lipschitz or linear settings can be solved with a simple pigeonhole principle, the argument for the coarse case required the additional set theoretic assumptions  mentioned above.

\begin{prop}[MA+$\neg$CH]\label{prop:equi-coarse}
If $(X_\alpha, d_\alpha)_{\alpha<\omega_1}$ is a collection of metric spaces such that for all $\alpha<\omega_1$, $X_\alpha$ embeds coarsely into a metric space $(M,d)$, then there exists an uncountable subset $C$ of $\omega_1$ such that $(X_\alpha, d_\alpha)_{\alpha\in C}$ embeds equi-coarsely into $(M,d)$.
\end{prop}

In Section \ref{sec:emb-ranks}, we recall the definitions of the objects being considered in this article (e.g. coarse, uniform, bi-Lipschitz embeddings) and introduce various ordinal indices related to the corresponding notions of metric embeddability.
Our point of view relies on the notion of rank of well-founded binary relations on Polish spaces (instead of trees on Polish spaces as in \cite{BLMS_FM}). In order to avoid the use of additional axioms through Proposition \ref{prop:equi-coarse}, we incorporate the main characteristics of the metric embeddings into the definitions of the ordinal indices. The guaranteed uniform control on the moduli in the conclusion of Proposition \ref{prop:equi-coarse} becomes an automatic feature of the definition and consequently the use of additional set theoretic axioms is bypassed.
In Section \ref{sec:simple-emb-ranks}, we introduce some natural and easily defined embeddability ranks. However, estimating these ranks does not seem straightforward in practice and therefore, in Section \ref{sec:schauder-ranks}, we define embeddability ranks induced by Schauder bases. These ranks are similar to the nonlinear indices introduced in \cite{BLMS_FM} but are more general and most importantly, as hinted at above, incorporate the compression and expansion moduli into their definitions. As in \cite{BLMS_FM}, these ranks are easily estimated and lead to the same results without the use of additional set theoretic assumptions.
 
The extraction of subcollections of embeddings with uniform control on their moduli of expansion and compression is an interesting problem in its own right and, in the last section, we generalize Proposition \ref{prop:equi-coarse} by relating this type of problems to certain cardinal invariants of the continuum, in particular the bounding number $\mathfrak{b}$. Our proof is arguably more transparent than the one in \cite{BLMS_FM}.

%%%%%%%%%%%%%%%%%%%%%%%%%%%%%%%%%%%%%%%%%%%%%%%%%%%%%%%%%%%
%%%%%%%%%%%%%%%%%%%%%%%%%%%%%%%%%%%%%%%%%%%%%%%%%%%%%%%%%%%
%%%%%%%%%%%%%%%%%%%%%%%%%%%%%%%%%%%%%%%%%%%%%%%%%%%%%%%%%%%
%%%%%%%%%%%%%%%%%%%%%%%%%%%%%%%%%%%%%%%%%%%%%%%%%%%%%%%%%%%
%%%%%%%%%%%%%%%%%%%%%%%%%%%%%%%%%%%%%%%%%%%%%%%%%%%%%%%%%%%
%%%%%%%%%%%%%%%%%%%%%%%%%%%%%%%%%%%%%%%%%%%%%%%%%%%%%%%%%%%
%%%%%%%%%%%%%%%%%%%%%%%%%%%%%%%%%%%%%%%%%%%%%%%%%%%%%%%%%%%
%%%%%%%%%%%%%%%%%%%%%%%%%%%%%%%%%%%%%%%%%%%%%%%%%%%%%%%%%%%
%%%%%%%%%%%%%%%%%%%%%%%%%%%%%%%%%%%%%%%%%%%%%%%%%%%%%%%%%%%
%%%%%%%%%%%%%%%%%%%%%%%%%%%%%%%%%%%%%%%%%%%%%%%%%%%%%%%%%%%

\section{Embeddability ranks}\label{sec:emb-ranks}

\subsection{Functional moduli}
The various notions of embeddability we will consider can all be described in terms of functional moduli, which we now proceed to define. For this purpose, let 
$$
\ku Z=\{\ldots, \tfrac 13, \tfrac 12, 1, 2,3, \ldots\}
$$
and set 
$$
\ku N=\{\kappa\colon \ku Z \to [0,\infty]\del \kappa \text{ is non-decreasing } \}.
$$
Observe that $\ku N$ can be seen as a closed subset of the Polish space $[0,\infty]^\ku Z$, where we view  $[0,\infty]$ as  the one-point compactification of $[0,\infty)$. 

When $\phi\colon X\to E$ is a function between two non-empty metric spaces, we define the compression modulus $\kappa_\phi\in \ku N$ and the expansion modulus $\om_\phi\in \ku N$ by
$$
\kappa_\phi(r)= \inf\{d_E\big(\phi(x),\phi(y)\big)\del d_X(x,y)\geqslant r\},\quad r\in \ku Z
$$
and
$$
\om_\phi(r)= \sup\{d_E\big(\phi(x),\phi(y)\big)\del d_X(x,y)\leqslant r\},\quad  r\in \ku Z.
$$
Remark that, if $X$ is bounded, there may be no $x,y\in X$ for which $d_X(x,y)\geqslant r$, which means that $\kappa_\phi(r)=\inf\tom=\infty$. So, for any $r\in \ku Z$, 
$$
\kappa_\phi(r)<\infty \;\equi\; \e x,y\in X\;\; d_X(x,y)\geqslant r.
$$
We need one more piece of notation. For any real number $t>0$,  let $t_-$ and $t_+$ be respectively the largest and smallest element of $\ku Z$ so that
$$
t_-\leqslant t\leqslant t_+
$$
and observe that, for all $t>0$, we have $\frac t2\leqslant t_-\leqslant t\leqslant t_+\leqslant 2t$. 
By definition of $\kappa_\phi$ and $\om_\phi$,  we note that
$$
\kappa_\phi\big( d(x,y)_-\big)\leqslant d\big(\phi(x),\phi(y)\big)\leqslant\om_\phi \big( d(x,y)_+\big) 
$$
for all $x,y\in X$.

As noted above, the various notions of embeddability between metric spaces we shall consider are defined by  imposing different restrictions of the compression and expansion moduli of a function $\phi\colon X\to E$. We therefore define 
$$
\ku M_{\text{coarse}}=
\{(\kappa,\om)\in \ku N\times \ku N\del \om(r)<\infty\text{ for all }r\in \ku Z \text{ and }\lim_{n\to \infty}\kappa(n)=\infty\},
$$
$$
\ku M_{\text{uniform}}=
\{(\kappa,\om)\in \ku N\times \ku N\del 0<\kappa(r)\text{ for all }r\in \ku Z \text{ and }\lim_{n\to \infty}\om(\tfrac1n)=0\}
$$
and finally
$$
\ku M_{\text{Lipschitz}}=\{(\kappa,\om)\in\ku N\times \ku N\del  \kappa(r)\geqslant cr,\;  \om(r)\leqslant Cr \text{ for some  }0<c<C<\infty\}.
$$
A map $\phi\colon X\to E$ is then called a {\em coarse}, {\em uniform} or {\em bi-Lipschitz embedding} according to whether $(\kappa_\phi,\om_\phi)$ belongs to $\ku M_{\text{coarse}}$, $\ku M_{\text{uniform}}$ or to $\ku M_{\text{Lipschitz}}$.

%%%%%%%%%%%%%%%%%%%%%%%%%%%%%%%%%%%%%%%%%%%%%%%%%%%%%%%%%%%%%%%%%%%%%
%%%%%%%%%%%%%%%%%%%%%%%%%%%%%%%%%%%%%%%%%%%%%%%%%%%%%%%%%%%%%%%%%%%%%
%%%%%%%%%%%%%%%%%%%%%%%%%%%%%%%%%%%%%%%%%%%%%%%%%%%%%%%%%%%%%%%%%%%%%
%%%%%%%%%%%%%%%%%%%%%%%%%%%%%%%%%%%%%%%%%%%%%%%%%%%%%%%%%%%%%%%%%%%%%
%%%%%%%%%%%%%%%%%%%%%%%%%%%%%%%%%%%%%%%%%%%%%%%%%%%%%%%%%%%%%%%%%%%%%
%%%%%%%%%%%%%%%%%%%%%%%%%%%%%%%%%%%%%%%%%%%%%%%%%%%%%%%%%%%%%%%%%%%%%
%%%%%%%%%%%%%%%%%%%%%%%%%%%%%%%%%%%%%%%%%%%%%%%%%%%%%%%%%%%%%%%%%%%%%
%%%%%%%%%%%%%%%%%%%%%%%%%%%%%%%%%%%%%%%%%%%%%%%%%%%%%%%%%%%%%%%%%%%%%
%%%%%%%%%%%%%%%%%%%%%%%%%%%%%%%%%%%%%%%%%%%%%%%%%%%%%%%%%%%%%%%%%%%%%
%%%%%%%%%%%%%%%%%%%%%%%%%%%%%%%%%%%%%%%%%%%%%%%%%%%%%%%%%%%%%%%%%%%%%
%%%%%%%%%%%%%%%%%%%%%%%%%%%%%%%%%%%%%%%%%%%%%%%%%%%%%%%%%%%%%%%%%%%%%

\subsection{Simple embeddability ranks}\label{sec:simple-emb-ranks}
Suppose $X$ and $E$  are metric spaces and that $D=\{x_1,x_2, \ldots\}$ is a fixed enumeration of a countable dense subset of $X$. We then let $\Omega\big({D,E}\big) $ be the collection of all quadruples 
$(\kappa,\om, k,\phi)  \in    \ku N\times \ku N \times \N \times E^{D}$ so that
$$
\kappa\big( d(x,y)_-\big)\leqslant d\big(\phi(x),\phi(y)\big)\leqslant\om \big( d(x,y)_+\big) 
$$
for all distinct $x,y\in \{x_1,x_2, \ldots, x_k\}$. Let also
$$
{\mathcal Coa}\big({D,E}\big)=\{(\kappa,\om, k,\phi)  \in  \Omega\big({D,E}\big) \del (\kappa,\om)\in \ku M_{\text{coarse}}\},
$$
$$
{\mathcal Uni}\big({D,E}\big)=\{(\kappa,\om, k,\phi)  \in  \Omega\big({D,E}\big) \del (\kappa,\om)\in \ku M_{\text{uniform}}\},
$$
$$
{\mathcal Lip}\big({D,E}\big)=\{(\kappa,\om, k,\phi)  \in  \Omega\big({D,E}\big) \del (\kappa,\om)\in \ku M_{\text{Lipschitz}}\}.
$$
Finally, we define a binary relation $<$ on $\Omega\big({D,E}\big) $ by letting 
\maths{
(\kappa, \om, k, \phi) <&(\kappa', \om' , k', \phi') \;\equi\; \\
& (\kappa,\om)=(\kappa',\om')\;\;\&\;\; k=k'+1 \;\;\&\;\;\phi\upharpoonright_{\{x_1,x_2, \ldots, x_{k'}\}}=\phi'\upharpoonright_{\{x_1,x_2, \ldots, x_{k'}\}},
}
where $\psi \upharpoonright_{A}$ denotes the restriction of a map $\psi$ to a subset $A$.

Recall that a binary relation $\prec$ on a set $A$ is said to be  {\em ill-founded} if and only if  there is an infinite sequence $a_1,a_2,\ldots\in A$ so that 
$$
\cdots\prec a_3\prec a_2\prec a_1.
$$
On the other hand, if there is no such infinite sequence, we say that $\prec$ is {\em well-founded}. Observe that another way of stating that $\prec$ is well-founded is by saying that every non-empty subset $B\subseteq A$ has a {\em minimal} element $x\in B$, i.e., so that $y\not\prec x$ for all $y\in B$.

\begin{lemme}\label{lem:elementary}
The binary relation $<$ is ill-founded on any of the spaces 
$$
{\mathcal Coa}\big({D,E}\big),\qquad 
{\mathcal Uni}\big({D,E}\big),
\qquad 
{\mathcal Lip}\big({D,E}\big)
$$ if and only if  $X$ coarsely embeds, uniformly embeds, respectively  bi-Lipschitz embeds into $E$.
\end{lemme}

\begin{proof}
We give the proof for the case of coarse embeddings, the two other cases being entirely analogous. 
So suppose first that $\phi \colon X\to E$ is a coarse embedding and hence that $(\kappa_\phi,\om_\phi)\in \ku M_{\text{coarse}}$.
Furthermore, 
$$
\kappa_\phi\big( d(x,y)_-\big)\leqslant d\big(\phi(x),\phi(y)\big)\leqslant\om_\phi \big( d(x,y)_+\big)
$$
for all distinct $x,y\in X$. Letting $\psi=\phi\upharpoonright_D$, we have $(\kappa_\phi,\om_\phi,k,\psi)\in {\mathcal Coa}\big({D,E}\big)$ for all $k$ and 
$$
\dots < (\kappa_\phi,\om_\phi,  3, \psi) < (\kappa_\phi,\om_\phi,  2, \psi) < (\kappa_\phi,\om_\phi, 1, \psi)
$$
is an infinite descending sequence in ${\mathcal Coa}\big({D,E}\big)$, showing  ill-foundedness.

Conversely, suppose that $<$ is ill-founded on the set ${\mathcal Coa}\big({D,E}\big)$. This means that there is some $(\kappa,\om)\in \ku M_{\text{coarse}}$, a natural number $k_0$, and functions $\phi_k\colon D\to E$ for $k\geqslant k_0$ so that, for all $k\geqslant k_0$, 
$$
\phi_{k+1} \upharpoonright_{\{x_1,x_2, \ldots, x_{k}\}}=\phi_k\upharpoonright_{\{x_1,x_2, \ldots, x_{k}\}}
$$
and 
$$
\kappa\big( d_X(x,y)_-\big)\leqslant d_E\big(\phi_k(x),\phi_k(y)\big)\leqslant\om \big( d_X(x,y)_+\big)
$$
for all distinct $x,y\in \{x_1,x_2, \ldots, x_{k}\}$. It follows that there is a single function $\phi\colon D\to E$ extending the restrictions $\phi_k\upharpoonright_{\{x_1,x_2, \ldots, x_{k}\}}$ and satisfying
$$
\kappa\big( d_X(x,y)_-\big)\leqslant d_E\big(\phi(x),\phi(y)\big)\leqslant\om \big( d_X(x,y)_+\big)
$$
for all distinct $x,y\in D$. In particular, for all $n\geqslant 1$ and $x,y\in D$, we have that
$$
d_X(x,y)\geqslant  n\;\saa\;  d_E\big(\phi(x),\phi(y)\big)\geqslant \kappa\big( d_X(x,y)_-\big) \geqslant \kappa (n)
$$
and 
$$
d_X(x,y)\leqslant n \;\saa\; d_E\big(\phi(x),\phi(y)\big) \leqslant \om \big( d_X(x,y)_+\big) \leqslant \om (n).
$$

Because $D$ is dense in $X$, we can now extend $\phi$ to a map $\phi\colon X\to E$ as follows. For $x\in X\setminus  D$, pick an arbitrary point $x'\in D$ so that $d_X(x',x)<\frac 12$ and set
$$
\phi(x)=\phi(x').
$$
Then we still have 
$$
d_X(x,y)\geqslant  n+1\;\saa\;  d_E\big(\phi(x),\phi(y)\big) \geqslant \kappa (n)
$$
and 
$$
d_X(x,y)\leqslant n-1 \;\saa\; d_E\big(\phi(x),\phi(y)\big) \leqslant \om (n).
$$
As  $\lim_{n\to \infty}\kappa(n)=\infty$ and $\om(n)<\infty$ for all $n$, it follows that $\phi$ is a coarse embedding. 
\end{proof}

Recall that, if $\prec$ is a well-founded relation on a set $A$, we may define an ordinal-valued rank function $\rho\colon A\to {\sf Ord}$ by letting
$$
\rho(x)=\sup\{\rho(y)+1\del y\prec x\}
$$
for all $x\in A$. Thus, $\rho(x)=0$ if and only if $x\in A$ is minimal, that is, if $y\not\prec x$ for all $y\in A$. Similarly, $\rho(x)\geqslant \omega$ if and only if, for all $m\geqslant 1$, one may find $x_1,\ldots,x_m\in A$ so that 
$$
x_1\prec x_2\prec \dots \prec x_m\prec x.
$$
The image $\rho[A]=\{\rho(x)\del x\in A\}$ of the rank function $\rho$ associated with a particular well-founded relation $\prec$ on a set $A$ is always an initial segment of the ordinals and hence is an ordinal itself. That allows us to define the {\em rank} by
$$
{\sf rk}(A,\prec)=\rho[A]=\sup\{\rho(x)+1\del x\in A\}\in {\sf Ord}.
$$

Let us now recall the well-known boundedness theorem for analytic well-founded relations (see \cite[Theorem 31.1]{Kechris}). For this, let us say that a binary relation $\prec$ on a Polish space $A$ is {\em analytic} if it is analytic when viewed as a subset of $A\times A$.

\begin{thm}\label{boundedness}
Suppose $\prec$ is an analytic, well-founded, binary relation on a  Borel subset of a Polish space $A$. Then the rank ${\sf rk}(A,\prec)$ is countable.
\end{thm}

The following theorem then connects Theorem \ref{boundedness} with Lemma \ref{lem:elementary}.

\begin{thm}\label{thm:elementary}
Suppose $X$ and $E$ are separable complete metric spaces and $D=\{x_1,x_2, \ldots\}$ is a fixed enumeration of a dense subset of $X$. Then either $X$ coarsely embeds into $E$ or the rank
    $$
    {\sf rk}\big({\mathcal Coa}\big({D,E}\big), <\big)
    $$
    is countable.
\end{thm}
Similar results hold for uniform and bi-Lipschitz embeddability, and our proof applies to all of the three cases.
    
\begin{proof}   
Because $D$ is countable and $E$ is Polish, the set
$\Omega\big(D,E\big)$ is a closed subset of the Polish space $\ku N\times \ku N\times \N\times E^{D}$. Also, $\ku M_{\text{coarse}}$, $\ku M_{\text{uniform}}$, and $\ku M_{\text{Lipschitz}}$, are all Borel sets in $\ku N\times \ku N$, whereby each of the sets 
    $$
{\mathcal Coa}\big({D,E}\big),\qquad 
{\mathcal Uni}\big({D,E}\big),
\qquad 
{\mathcal Lip}\big({D,E}\big)
$$
are Borel subsets of the Polish space $\ku N\times \ku N\times \N\times E^{D}$. Furthermore,  the relation $<$ is closed as a subset of $\Omega\big(D,E\big)$. The theorem now follows from Theorem \ref{boundedness} and Lemma \ref{lem:elementary}.
\end{proof}

%%%%%%%%%%%%%%%%%%%%%%%%%%%%%%%%%%%%%%%%%%%%%%%%%%%%%%%%%%%
%%%%%%%%%%%%%%%%%%%%%%%%%%%%%%%%%%%%%%%%%%%%%%%%%%%%%%%%%%%
%%%%%%%%%%%%%%%%%%%%%%%%%%%%%%%%%%%%%%%%%%%%%%%%%%%%%%%%%%%
%%%%%%%%%%%%%%%%%%%%%%%%%%%%%%%%%%%%%%%%%%%%%%%%%%%%%%%%%%%
%%%%%%%%%%%%%%%%%%%%%%%%%%%%%%%%%%%%%%%%%%%%%%%%%%%%%%%%%%%
%%%%%%%%%%%%%%%%%%%%%%%%%%%%%%%%%%%%%%%%%%%%%%%%%%%%%%%%%%%
%%%%%%%%%%%%%%%%%%%%%%%%%%%%%%%%%%%%%%%%%%%%%%%%%%%%%%%%%%%
%%%%%%%%%%%%%%%%%%%%%%%%%%%%%%%%%%%%%%%%%%%%%%%%%%%%%%%%%%%
%%%%%%%%%%%%%%%%%%%%%%%%%%%%%%%%%%%%%%%%%%%%%%%%%%%%%%%%%%%
%%%%%%%%%%%%%%%%%%%%%%%%%%%%%%%%%%%%%%%%%%%%%%%%%%%%%%%%%%%

\subsection{Embeddability ranks based on Schauder bases}\label{sec:schauder-ranks}

In this section we define embeddability ranks that are a little more delicate to define but are easier to estimate from below. 

Let $R$ be a ring and denote by $R[e_1,e_2,\ldots]$ the free $R$-module with basis elements $e_1,e_2, \ldots$. Also, if $A\subseteq \{e_1,e_2,\ldots\}$, we let $R[A]$ denote the submodule of $R[e_1,e_2,\ldots]$ spanned by $A$. For two finite subsets $A,B\subseteq \{e_1,e_2,\ldots\}$, we write 
$$
A\prec^\star B
$$ 
to denote that $B=\{e_{n_1},\ldots, e_{n_k}\}$ and $A=\{e_{n_1}, \ldots, e_{n_k}, e_{n_{k+1}}\}$ for some $n_1<\ldots<n_k<n_{k+1}$.

Suppose now that $(e_n)_{n=1}^\infty$ is a basic sequence in a real Banach space $(X,\norm{\cdot})$. Let also the ring $R$ be either $\Z$ or $\Q$, in which case, $R[e_1,e_2,\ldots]\subseteq X$. We thus view $R[e_1,e_2,\ldots]$ as a metric space with the metric induced by the norm on $X$. Assume that $E$ is a metric space and define $\Theta\big(R,(e_n),E)$ as the collection of all quadruples $(\kappa,\omega, A, \phi)$ where $\kappa, \omega\in \ku N$, $A$ is a finite subset of $\{e_1,e_2,\ldots\}$ and  $\phi$ is a function $R[e_1,e_2,\ldots]\maps \phi E$ so that
$$
\kappa\big( \norm{x-y}_-\big)\leqslant d_E\big(\phi(x),\phi(y)\big)\leqslant\om \big( \norm{x-y}_+\big)
$$
for all distinct $x,y\in R[A]$. As in the preceding section, we let
$$
{\mathcal Coa}\big({R,(e_n),E}\big)=\{(\kappa,\om, A,\phi)  \in  \Theta\big(R,(e_n),E\big) \del (\kappa,\om)\in \ku M_{\text{coarse}}\},
$$
$$
{\mathcal Uni}\big({R,(e_n),E}\big)=\{(\kappa,\om, A,\phi)  \in \Theta\big(R,(e_n),E\big) \del (\kappa,\om)\in \ku M_{\text{uniform}}\},
$$
$$
{\mathcal Lip}\big({R,(e_n),E}\big)=\{(\kappa,\om, A,\phi)  \in \Theta\big(R,(e_n),E\big) \del (\kappa,\om)\in \ku M_{\text{Lipschitz}}\}
$$
and define a binary relation $<$ on $\Theta\big(R,(e_n),E\big)$ by setting
\maths{
(\kappa, \om, A, \phi) <(\kappa', \om' , A', &\phi') \;\equi\; \\
& (\kappa,\om)=(\kappa',\om')\;\;\&\;\; A\prec^\star A'  \;\;\&\;\;\phi\upharpoonright_{R[A']}=\phi'\upharpoonright_{R[A']}.
}
In analogy we Lemma \ref{lem:elementary}, we now have the following.

\begin{lemme}\label{lem:module}
The binary relation $<$ is ill-founded on any of the spaces 
$$
{\mathcal Coa}\big({R,(e_n),E}\big),\qquad 
{\mathcal Uni}\big({R,(e_n),E}\big),
\qquad 
{\mathcal Lip}\big({R,(e_n),E}\big)
$$ 
if and only if  there is an infinite subsequence $e_{n_1}, e_{n_2}, e_{n_3}, \ldots$ of the basis of $X$ so that $R[e_{n_1}, e_{n_2}, e_{n_3}, \ldots]$ coarsely embeds, uniformly embeds, respectively  bi-Lipschitzly embeds into $E$.
\end{lemme}

Now, when $(e_n)_{n=1}^\infty$ is a subsymmetric Schauder basis for $X$, then, for any infinite subsequence $e_{n_1}, e_{n_2}, e_{n_3}, \ldots$, the map $e_k\mapsto e_{n_k}$ extends to a bi-Lipschitz equivalence between the two metric spaces $R[e_{1}, e_{2}, e_{3}, \ldots]$ and $R[e_{n_1}, e_{n_2}, e_{n_3}, \ldots]$. Also, as $R$ is countable, so is $R[e_{1}, e_{2}, e_{3}, \ldots]$. The following result is thus obtained exactly as Theorem \ref{thm:elementary}.

\begin{thm}
Suppose  $(e_n)_{n=1}^\infty$ is a subsymmetric Schauder basis for a Banach space $X$, $E$ is a separable complete metric space and $R$ is either $\Z$ or $\Q$.  Then either $R[e_{1}, e_{2}, e_{3}, \ldots]$ coarsely embeds into $E$ or the rank
    $$
    {\sf rk}\big({\mathcal Coa}\big({R,(e_n),E}\big), <\big)
    $$
    is countable.
\end{thm}
Again, similar results hold for uniform and bi-Lipschitz embeddability.

Observe now that, if $R=\Q$, then  $R[e_{1}, e_{2}, e_{3}, \ldots]$ is a dense subspace of the Banach space $X$ and so coarse, uniform or bi-Lipschitz embeddability of $X$ into $E$ is equivalent to the corresponding embeddability of  $R[e_{1}, e_{2}, e_{3}, \ldots]$ into $E$. We thus have the following corollary.
\begin{cor}
Suppose  $(e_n)_{n=1}^\infty$ is a subsymmetric Schauder basis for a Banach space $X$ and $E$ is a separable complete metric space.  Then either $X$ coarsely embeds into $E$ or the rank
    $$
    {\sf rk}\big({\mathcal Coa}\big({\Q,(e_n),E}\big), <\big)
    $$
    is countable.
\end{cor}

In the specific case when $(e_n)_{n=1}^\infty$ is the standard unit vector basis for $X=c_0$, we note that $\Z[e_{1}, e_{2}, e_{3}, \ldots]$ is $1$-dense in $c_0$ and hence that coarse embeddability of $c_0$ into $E$ is again equivalent to that of $\Z[e_{1}, e_{2}, e_{3}, \ldots]$ into $E$.
\begin{cor}
Suppose  $(e_n)_{n=1}^\infty$ is the standard unit vector basis for $c_0$ and $E$ is a separable complete metric space.  Then either $c_0$ coarsely embeds into $E$ or the rank
    $$
    {\sf rk}\big({\mathcal Coa}\big({\Z,(e_n),E}\big), <\big)
    $$
    is countable.
\end{cor}

%%%%%%%%%%%%%%%%%%%%%%%%%%%%%%%%%%%%%%%%%%%%%%%%%%%%%%%%%%%
%%%%%%%%%%%%%%%%%%%%%%%%%%%%%%%%%%%%%%%%%%%%%%%%%%%%%%%%%%%
%%%%%%%%%%%%%%%%%%%%%%%%%%%%%%%%%%%%%%%%%%%%%%%%%%%%%%%%%%%
%%%%%%%%%%%%%%%%%%%%%%%%%%%%%%%%%%%%%%%%%%%%%%%%%%%%%%%%%%%
%%%%%%%%%%%%%%%%%%%%%%%%%%%%%%%%%%%%%%%%%%%%%%%%%%%%%%%%%%%
%%%%%%%%%%%%%%%%%%%%%%%%%%%%%%%%%%%%%%%%%%%%%%%%%%%%%%%%%%%
%%%%%%%%%%%%%%%%%%%%%%%%%%%%%%%%%%%%%%%%%%%%%%%%%%%%%%%%%%%
%%%%%%%%%%%%%%%%%%%%%%%%%%%%%%%%%%%%%%%%%%%%%%%%%%%%%%%%%%%
%%%%%%%%%%%%%%%%%%%%%%%%%%%%%%%%%%%%%%%%%%%%%%%%%%%%%%%%%%%
%%%%%%%%%%%%%%%%%%%%%%%%%%%%%%%%%%%%%%%%%%%%%%%%%%%%%%%%%%%

\subsection{Schreier spaces}
To provide lower estimates for the coarse embeddability ranks, as in \cite{BLMS_FM} we will employ metric spaces built over the so-called Schreier sets, which are compact hereditary families of finite subsets of $\N=\{1,2,3,\ldots\}$. For this, if $A$ and $B$ are finite subsets of $\N$ and $n\in \N$, we write
$$
n\leqslant A
$$
to denote that either $A$ is empty or that $n\leqslant \min A$ and write
$$
A<B
$$
to mean that either one of $A$ and $B$ is empty or that $ \max A<\min B$. Finally, let
$$
A\prec^\star B
$$ 
denote that $B=\{{n_1},\ldots, {n_k}\}$ and $A=\{{n_1}, \ldots, {n_k}, {n_{k+1}}\}$ for some $n_1<\ldots<n_k<n_{k+1}$.
For every countable limit ordinal $\lambda$, we fix once and for all an increasing sequence $\lambda_1<\lambda_2<\lambda_3<\ldots$ of ordinals with limit $\lambda$.

By induction on $\alpha<\om_1$, we define a family $\ku S_\alpha\subseteq \ku P(\N)$ consisting exclusively of finite subsets of $\N$ so that $\ku S_\alpha$ is {\em hereditary}, i.e., if $A\subseteq B$ and $B\in \ku S_\alpha$, then $A\in \ku S_\alpha$, and so that $\ku S_\alpha$ is compact when viewed as a subset of the Cantor space $\{0,1\}^\N$. This is done as follows.
\maths{
\ku S_0&=\big\{\{n\}\del n\in \N\big\}\cup \{\tom\},\\
\ku S_{\alpha+1}&=\big\{ \bigcup_{j=1}^nA_j\del n\in \N \;\;\&\;\; A_1,\ldots, A_n\in \ku S_\alpha\;\;\&\;\; n\leqslant A_1<\ldots< A_n\big\},\\
\ku S_{\lambda}&=\big\{ A\del \e n\in \N \;\;\; A \in \ku S_{\lambda_n}\;\&\; n\leqslant A\big\} \text{ for $\lambda$ a limit ordinal}.
}
By transfinite induction on $\alpha<\om_1$, one can show that 
$$
{\sf rk}\big(\ku S_\alpha, \prec^\star\big)=\om^\alpha+1.
$$

If now $(e_n)_{n=1}^\infty$ is a Schauder basis for a Banach space $X$, we let 
$$
\Z[\ku S_\alpha]=\Big\{\sum_{i\in A}t_ie_i\Del A\in \ku S_\alpha\;\;\&\;\; t_i\in \Z\Big\}.
$$
Note that, contrary to the subspaces $\Z[A]$ for $A\subseteq \{e_1,e_2,\ldots\}$, the subset $\Z[\ku S_\alpha]$ will not be a submodule of $\Z[e_1,e_2,\ldots]$ as it is not closed under addition. We write $\Z[\ku S_\alpha]_X$ for the metric space obtained by endowing $\Z[\ku S_\alpha]$ with the distance induced by the norm of $X$. 

\begin{thm}
Suppose  $(e_n)_{n=1}^\infty$ is the standard unit vector basis for $c_0$ and $E$ is a separable complete metric space.  Then $c_0$ coarsely embeds into $E$ if and only if $\Z[\ku S_\alpha]_{c_0}$ coarsely embeds into $E$ for all $\alpha<\om_1$.
\end{thm}

\begin{proof}
Suppose $\alpha<\om_1$ and that $\Z[\ku S_\alpha]_{c_0}\maps\phi E$ is a coarse embedding and let $\kappa_\phi$ and $\om_\phi$ be the compression and expansion moduli of $\phi$. Extend also $\phi$ arbitrarily to a map 
$\Z[e_1,e_2,\ldots]\maps \phi E$. This means that, for all $A\in \ku S_\alpha$, we have 
$$
\kappa_\phi\big( \norm{x-y}_-\big)\leqslant d_E\big(\phi(x),\phi(y)\big)\leqslant\om_\phi \big( \norm{x-y}_+\big)
$$
whenever $x,y\in \Z[\{e_i\}_{i\in A}]$ are distinct and so 
$$
\big(\kappa_\phi,\om_\phi, \{e_i\}_{i\in A}, \phi\big)\in {\mathcal Coa}\big({\Z,(e_n),E}\big).
$$
Furthermore, we note that the map $\ku S_\alpha\maps f {\mathcal Coa}\big({\Z,(e_n),E}\big)$ given by
$$
f(A)=\big(\kappa_\phi,\om_\phi, \{e_i\}_{i\in A}, \phi\big)
$$
satisfies 
$$
A\prec^\star B\saa f(A)<f(B),
$$
which implies that
$$
\om^\alpha+1={\sf rk}\big(\ku S_\alpha, \prec^\star\big)\leqslant {\sf rk}\big({\mathcal Coa}\big({\Z,(e_n),E}\big), <\big).
$$
Therefore, if $\Z[\ku S_\alpha]_{c_0}$ coarsely embeds into $E$ for all $\alpha<\om_1$, we see that the rank ${\sf rk}\big({\mathcal Coa}\big({\Z,(e_n),E}\big), <\big)$ is uncountable and therefore that $c_0$ coarsely embeds into $E$.  The reverse implication is immediate.
\end{proof}

%%%%%%%%%%%%%%%%%%%%%%%%%%%%%%%%%%%%%%%%%%%%%%%%%%%%%%%%%%%
%%%%%%%%%%%%%%%%%%%%%%%%%%%%%%%%%%%%%%%%%%%%%%%%%%%%%%%%%%%
%%%%%%%%%%%%%%%%%%%%%%%%%%%%%%%%%%%%%%%%%%%%%%%%%%%%%%%%%%%
%%%%%%%%%%%%%%%%%%%%%%%%%%%%%%%%%%%%%%%%%%%%%%%%%%%%%%%%%%%
%%%%%%%%%%%%%%%%%%%%%%%%%%%%%%%%%%%%%%%%%%%%%%%%%%%%%%%%%%%
%%%%%%%%%%%%%%%%%%%%%%%%%%%%%%%%%%%%%%%%%%%%%%%%%%%%%%%%%%%
%%%%%%%%%%%%%%%%%%%%%%%%%%%%%%%%%%%%%%%%%%%%%%%%%%%%%%%%%%%
%%%%%%%%%%%%%%%%%%%%%%%%%%%%%%%%%%%%%%%%%%%%%%%%%%%%%%%%%%%
%%%%%%%%%%%%%%%%%%%%%%%%%%%%%%%%%%%%%%%%%%%%%%%%%%%%%%%%%%%
%%%%%%%%%%%%%%%%%%%%%%%%%%%%%%%%%%%%%%%%%%%%%%%%%%%%%%%%%%%

\section{A remark on extracting equi-coarse embeddings}

If $A$ and $B$ are two subsets of $\N$, we say that $A$ is \emph{almost contained} in $B$, written $A\subseteq^* B$, if $|A\setminus B|<\aleph_0$. A infinite set $D\subseteq \N$ is called a \emph{pseudo-intersection} of a family $\mathcal{C}$ of subsets of $\N$ provided that $D\subseteq^* C$ for all $C \in \mathcal{C}$. In \cite{BLMS_FM}, Proposition \ref{prop:equi-coarse} was established via the following diagonalization principle.

\begin{lemme}[MA+$\neg$CH]\label{lem:tower}
Let $\{A_\alpha\}_{\alpha<\omega_1}$ be a family of infinite subsets of $\N$ such that $A_\beta\subseteq^* A_\alpha$ whenever $\alpha<\beta$. Then $\{A_\alpha\}_{\alpha<\omega_1}$ has a pseudo-intersection.
\end{lemme}

Given an ordinal $\lambda$, a \emph{$\lambda$-tower} is a family $\{A_\alpha\}_{\alpha<\lambda}$ of infinite subsets of $\N$ with no pseudo-intersection and such that $A_\beta\subseteq^* A_\alpha$ whenever $\alpha<\beta<\lambda$. We then define the cardinal number $\go t$ by
$$
\mathfrak{t}=\min\{|\lambda|\colon \text{there exists a } \lambda\text{-tower}\}.
$$
Whereas it is easy to see that  $\aleph_1\leqslant\mathfrak{t}\leqslant 2^{\aleph_0}$, the specific value of  $\go t$ is highly sensitive to additional set theoretical assumptions. For example, under the assumption of  $\text{MA}_{\sigma\text{-centered}}$ (a slight weakening of MA),  we have 
$$
\mathfrak{t}=2^{\aleph_0}
$$ 
(see \cite[Theorem 19.24, Theorem 19.25]{JustWeese}), which immediately implies Lemma \ref{lem:tower}. Therefore, the argument from \cite{BLMS_FM} could be straightforwardly extended to other uncountable cardinals strictly less than $2^{\aleph_0}$. Here we present a different way to derive Proposition \ref{prop:equi-coarse} in terms of another cardinal number, the bounding number $\mathfrak{b}$, that is more naturally related to the problem at hand.

For functions $f,g\colon \N\to \N$,  set
$$
f<^*g\;\equi \; \text{there is some $n$ so that }f(m)<g(m) \text{ for all } m\geqslant n.
$$
In this case, we say that $g$ {\em eventually dominates} $f$.
The relation $<^*$ defines a partial order on $\N^\N$ and we say that a family $F\subseteq \N^\N$ is \emph{bounded} (with respect to $<^*$) if there exists $g\in \N^\N$ such that $f<^* g$ for all $f\in F$. If a family is not bounded it is called an \emph{unbounded} family. The \emph{bounding number} is the cardinal number $\mathfrak{b}$ defined by
$$
\mathfrak{b}=\min\{ |F|\colon F\subseteq \N^\N\; \&\;  F \text{ is unbounded}\}. 
$$

The following inequalities hold without any additional assumptions (see \cite[Theorem 19.24]{JustWeese}),
$$
\aleph_1\leqslant \go t\leqslant \go b\leqslant 2^{\aleph_0}.
$$
In particular, $\text{MA}_{\sigma\text{-centered}}$ and $\neg\,\text{CH}$ imply that $\mathfrak{b} = 2^{\aleph_0}>\aleph_1$.

The next lemma relates  the bounding number to the boundedness problem for families of non-negative maps and will subsequently be applied to the families of compression and expansion moduli of a family of coarse embeddings.

\begin{lemme}[$\go b>\aleph_1$]\label{prop:expansion}
Suppose $\{g_i\}_{i\in I}$ and $\{f_i\}_{i\in I}$ are uncountable families of non-decreasing functions $g_i,f_i\in\N^\N$ satisfying $\lim_{n\to \infty} g_i(n) = \lim_{n \to \infty} f_i(n)=\infty$.  Then there are non-decreasing functions $f,g\in \N^\N$ also satisfying  $\lim_{n\to \infty} g(n) = \lim_{n\to\infty} f(n)=\infty$ and an uncountable subset $J\subseteq I$ so that
$$
g\leqslant g_i, \quad f_i\leqslant f \quad\text{ for all $i\in J$}.
$$
\end{lemme}

\begin{proof}
Without loss of generality, we have $|I|=\aleph_1$.
In order to bound the $g_i$ from below we will bound the collection of generalized inverses $g^{\vee}_i$ from above.
Because $\lim_{n\to\infty} g_i(n)=\infty$, the generalized inverse $g^{\vee}_i\in \N^\N$ may be defined by
\begin{equation*}
g^{\vee}_i(k)=\min \{n \del g_i(n)>k\}
\end{equation*}
and hence satisfies
\begin{equation}\label{eq:g}
   g_i\big(g^\vee_i(k)\big)>k. 
\end{equation}
Also, because $|I|<\go b$, the families $F=\{f_i\}_{i\in I}$ and $G^\vee = \{g^{\vee}_i\}_{i\in I}$ are bounded and we may therefore choose some $h\in \N^\N$ so that 
$$
g^{\vee}_i,  f_i<^* h \quad\text{for all $i\in I$}.
$$
Replacing, if necessary $h$ by the function $h^\uparrow$ defined by $h^\uparrow(n)=\max_{m\leqslant n} h(m)$, we may also suppose that $h$ is non-decreasing. Since clearly $\lim_{n\to\infty} h(n)=\infty$, we can take $f=h$ as the upper bound for the collection $F$. 

Consider the generalized inverse $h^{\vee}$ defined by
$$
h^{\vee}(k)=\min \{n \del h(n)>k\},
$$
and note that, for all sufficiently large $k$, we have
\begin{equation}\label{eq:h}
h\big(h^\vee(k)-1\big) \leqslant k. 
\end{equation}
Also,  $h^{\vee}$ is non-decreasing and $\lim_{n \to \infty} h^{\vee}(n)=\infty$. Define $g\in \N^\N$ by
$$
g(n)=\max \{1, h^\vee(n)-1\}.
$$
Then, given $i\in I$, pick some $p$ so that both
$$
g^\vee_i(k)< h(k)
$$
and (2) hold for all $k\geqslant p$. Then, for all $n\geqslant p$ large enough so that $g(n)>p$, we have by definition of $h^\vee$ that
$$
g^\vee_i( g(n) )< h\big( g(n)\big) = h\big(h^\vee(n)-1\big) \stackrel{\eqref{eq:h}}{\leqslant} n.
$$
and so, because $g_i$ is non-decreasing,
$$
g(n) \stackrel{\eqref{eq:g}}{<} g_i\big(g^\vee_i( g(n) \big) \leqslant  g_i(n).
$$

Now, for all $i\in I$, let $n_i$ be minimal so that
$$
g(n)< g_i(n), \quad f_i(n)< f(n)
$$
for all $n\geqslant n_i$. Since $I$ is uncountable, there is some uncountable set $J\subseteq I$ and some $m$ so that $n_i=m$ for all $i\in J$. By replacing $f$ and $g$ by the functions
$$
f'(n)=\max\{f(n), f(m)\}
$$
and 
$$
g'(n)=\begin{cases}
		1 &\text{ if }n<m,\\
		g(n)&\text{ if } n\geqslant m,	
	\end{cases}
$$
we can then ensure that $g\leqslant g_i$ and  $f_i\leqslant f$ for all $i\in J$.
\end{proof}

\begin{prop}[$\go b>\aleph_1$]
Suppose $Y$ is a metric space and  $\{X_i\}_{i\in I}$ is an uncountable family of metric spaces so that each $X_i$ coarsely embeds into $Y$. Then there is an uncountable subfamily $\{X_i\}_{i\in J}$ that equi-coarsely embeds into $Y$, i.e., so that all the $X_i$ with $i\in J$ embed into $Y$ with the same moduli $\kappa$ and $\om$.     
\end{prop}

\begin{proof}
For every $i\in I$, there are maps $f_i \colon X_i \to Y$ and non-decreasing functions $\kappa_i, \omega_i\colon [0,\infty)\to [0,\infty)$ so that $\lim_{t\to\infty}\kappa_i(t)=\infty$ and, for all $x_1,x_2\in X_i$,
$$
\kappa_i(d_{X_i}(x_1,x_2)) \leqslant d_Y(f(x_1),f_i(x_2))\leqslant \omega_i(d_{X_i}(x_1,x_2)).
$$
We apply Lemma \ref{prop:expansion} to the maps $g_i(n)= \lfloor \kappa_i(n) + 1 \rfloor $ and $f_i(n)= \lfloor \omega_i(n) + 1  \rfloor $, which gives us non-decreasing functions $g,f\in \N^\N$ so that $g\leqslant g_i$ and $f_i\leqslant f$ for an uncountable number of $i$. Setting 
$$
\kappa(t)=g(\lfloor t\rfloor)-1, \qquad \om(t)=f(\lceil t\rceil)
$$
for all $t\geqslant 1$ and $\kappa(t)=0$, $\om(t)=f(1)$ for $t\in [0,1)$,
we see that $\kappa\leqslant \kappa_i$ and $\om_i\leqslant \om$ for an uncountable set of $i$'s.
\end{proof}

%%%%%%%%%%%%%%%%%%%%%%%%%%%%%%%%%%%%%%%%%%%%%%%%%%%%%%%%%%%
%%%%%%%%%%%%%%%%%%%%%%%%%%%%%%%%%%%%%%%%%%%%%%%%%%%%%%%%%%%
%%%%%%%%%%%%%%%%%%%%%%%%%%%%%%%%%%%%%%%%%%%%%%%%%%%%%%%%%%%
%%%%%%%%%%%%%%%%%%%%%%%%%%%%%%%%%%%%%%%%%%%%%%%%%%%%%%%%%%%
%%%%%%%%%%%%%%%%%%%%%%%%%%%%%%%%%%%%%%%%%%%%%%%%%%%%%%%%%%%
%%%%%%%%%%%%%%%%%%%%%%%%%%%%%%%%%%%%%%%%%%%%%%%%%%%%%%%%%%%
%%%%%%%%%%%%%%%%%%%%%%%%%%%%%%%%%%%%%%%%%%%%%%%%%%%%%%%%%%%
%%%%%%%%%%%%%%%%%%%%%%%%%%%%%%%%%%%%%%%%%%%%%%%%%%%%%%%%%%%
%%%%%%%%%%%%%%%%%%%%%%%%%%%%%%%%%%%%%%%%%%%%%%%%%%%%%%%%%%%
%%%%%%%%%%%%%%%%%%%%%%%%%%%%%%%%%%%%%%%%%%%%%%%%%%%%%%%%%%%

\bibliographystyle{alpha}
%\bibliography{../../../0_LaTex/biblio/flo_bib}


%%%%%%%%%%%%%%%%%%%%%%%%%%%%%%%%%%%%%%%%%%%%%%%%%%%%%%%%%%%
%%%%%%%%%%%%%%%%%%%%%%%%%%%%%%%%%%%%%%%%%%%%%%%%%%%%%%%%%%%
%%%%%%%%%%%%%%%%%%%%%%%%%%%%%%%%%%%%%%%%%%%%%%%%%%%%%%%%%%%
%%%%%%%%%%%%%%%%%%%%%%%%%%%%%%%%%%%%%%%%%%%%%%%%%%%%%%%%%%%
%%%%%%%%%%%%%%%%%%%%%%%%%%%%%%%%%%%%%%%%%%%%%%%%%%%%%%%%%%%
%%%%%%%%%%%%%%%%%%%%%%%%%%%%%%%%%%%%%%%%%%%%%%%%%%%%%%%%%%%
%%%%%%%%%%%%%%%%%%%%%%%%%%%%%%%%%%%%%%%%%%%%%%%%%%%%%%%%%%%
%%%%%%%%%%%%%%%%%%%%%%%%%%%%%%%%%%%%%%%%%%%%%%%%%%%%%%%%%%%
%%%%%%%%%%%%%%%%%%%%%%%%%%%%%%%%%%%%%%%%%%%%%%%%%%%%%%%%%%%
%%%%%%%%%%%%%%%%%%%%%%%%%%%%%%%%%%%%%%%%%%%%%%%%%%%%%%%%%%%

%%%%%%%%%%%%%%%%%%%%%%%%%%%%%%%%%%%%%%%%%%%%%%%%%%%%%%%%%%%
%%%%%%%%%%%%%%%%%%%%%%%%%%%%%%%%%%%%%%%%%%%%%%%%%%%%%%%%%%%
%%%%%%%%%%%%%%%%%%%%%%%%%%%%%%%%%%%%%%%%%%%%%%%%%%%%%%%%%%%
%%%%%%%%%%%%%%%%%%%%%%%%%%%%%%%%%%%%%%%%%%%%%%%%%%%%%%%%%%%
%%%%%%%%%%%%%%%%%%%%%%%%%%%%%%%%%%%%%%%%%%%%%%%%%%%%%%%%%%%
%%%%%%%%%%%%%%%%%%%%%%%%%%%%%%%%%%%%%%%%%%%%%%%%%%%%%%%%%%%
%%%%%%%%%%%%%%%%%%%%%%%%%%%%%%%%%%%%%%%%%%%%%%%%%%%%%%%%%%%
%%%%%%%%%%%%%%%%%%%%%%%%%%%%%%%%%%%%%%%%%%%%%%%%%%%%%%%%%%%

\end{document}